\documentclass{amsart}

\usepackage[usenames]{color}

\usepackage{amsmath}
\usepackage{amssymb}
\usepackage{verbatim}
\usepackage{enumerate}
\usepackage[all]{xy}
\usepackage[mathscr]{eucal}

\usepackage{amsthm}

\usepackage{amscd}
\usepackage{amsfonts}
\usepackage{latexsym}

\usepackage{amscd}

\usepackage{graphicx}

\newtheorem{theorem}{Theorem}[section]
\newtheorem{lemma}[theorem]{Lemma}
\newtheorem{proposition}[theorem]{Proposition}
\newtheorem{corollary}[theorem]{Corollary} 
\theoremstyle{definition}  
\newtheorem{definition}[theorem]{Definition}

\newtheorem{remark}[theorem]{Remark}

\newcommand{\Tr}{\text{Tr}}
\newcommand{\id}{\text{id}}

\newcommand{\End}{\text{End}} 
 
\newcommand{\Spec}{\text{Spec}} 
\renewcommand{\Vec}{\operatorname{Vec}}

\newcommand{\Hom}{\operatorname{Hom}} 
 
\def\uRep{\underline{\operatorname{Re}}\!\operatorname{p}} 

\newcommand{\Aut}{\text{Aut}}

\newcommand{\Rep}{\operatorname{Rep}}

\newcommand{\uRes}{\underline{\operatorname{Res}}}

\newcommand{\eps}{\varepsilon}

\newcommand{\B}{\mathcal{B}}
\newcommand{\T}{\mathcal{T}}
\newcommand{\C}{\mathcal{C}}
\newcommand{\D}{\mathcal{D}}

\newcommand{\A}{\mathcal{A}}

\newcommand{\K}{\mathcal{K}}

\newcommand{\be}{\mathbf{1}}

\renewcommand{\be}{\mathbf{1}}

\newcommand{\bI}{{\bf I}}
\newcommand{\BZ}{{\mathbb Z}}

\newcommand{\BR}{{\mathbb R}}

\newcommand{\ot}{\otimes}

\newcommand{\Z}{\mathbb{Z}}
\newcommand{\cat}{\mathcal}

\newcommand{\arup}[1]{\stackrel{#1}{\longrightarrow}}

\begin{document}

\title{On Deligne's category $\uRep^{ab}(S_d)$.}

\author{Jonathan Comes}
\address{J.C.: Department of Mathematics,
University of Oregon, Eugene, OR 97403, USA}
\email{jcomes@math.uoregon.edu}
\dedicatory{Dedicated to the memory of Andrei Zelevinsky}

\author{Victor Ostrik}
\address{V.O.: Department of Mathematics,
University of Oregon, Eugene, OR 97403, USA}
\email{vostrik@math.uoregon.edu}

\begin{abstract}
We prove a universal property of Deligne's category $\uRep^{ab}(S_d)$.  Along the way, we classify tensor ideals in the category $\uRep(S_d)$.
\end{abstract}

\date{\today}
\maketitle  



\section{Introduction}
\subsection{}
Let $F$ be a field of characteristic zero and let $I$ be a finite set.
Let $S_I$ be the symmetric group of the permutations of $I$ and
let $\Rep(S_I)$ be the category of finite dimensional $F-$linear representations of $S_I$
considered as a symmetric tensor category.
Let $X_I\in \Rep(S_I)$ be the space of $F-$valued functions on $I$
with an obvious action of $S_I$. 
The object $X_I$ with pointwise operations has a natural structure of associative commutative 
algebra with unit  $1_{X_I}$
in the category $\Rep(S_I)$. We have a morphism $\Tr : X_I \to F$ defined as a trace
of the operator of left multiplication; clearly the map $X_I \ot X_I \to F$ given by $x\otimes y\mapsto \Tr(xy)$ is a
non-degenerate pairing. Finally, $\Tr (1_{X_I})=\dim (X_I)=|I|$ where $|I|\ge 0$ is the cardinality of $I$.

Now let $G$ be a finite group acting on $d-$dimensional associative commutative unital algebra
$T$ over $F$ such that the pairing $\Tr(xy)$ is non-degenerate. It is easy to see\footnote{Set 
$I$ to be the set of $F-$algebra homomorphisms $T\to \bar F$ where $\bar F$ is an algebraic closure
of $F$ and use an obvious homomorphism $G\to S_I$.} that there
exists a finite set $I$ with $|I|=d$ and an essentially unique tensor functor 
$F: \Rep(S_I)\to \Rep(G)$ such that $F(X_I)\simeq T$ (isomorphism of $G-$algebras); 
in this sense the category $\Rep(S_I)$ is universal category (in the realm of representation
categories of finite groups) with object $X_I$ as above.

\subsection{}\label{Tabc}
Now for arbitrary symmetric tensor category $\T$ one can consider objects $T\in \T$ satisfying the
following:

(a) $T$ has a structure of associative commutative algebra 
(given by the multiplication map $\mu_T : T\ot T\to T$) with unit (given by the map $1_T: \be \to T$);

(b) The object $T$ is rigid. Moreover if we define the map $\Tr: T\to \be$ as a composition
$$T\arup{\id_T\ot coev_T} T\ot T\ot T^*\arup{\mu_T\ot\id_{T^\ast}} T\ot T^*\simeq T^*\ot T\arup{ev_T} \be,$$
then the pairing $T\otimes T\arup{\mu_T} T\arup{\Tr} \be$ is non-degenerate, that is it corresponds
to an isomorphism $T\simeq T^*$ under the identification $\Hom(T\otimes T,\be)=\Hom(T,T^*)$;

(c) We have $\dim(T)=t\in F$ (equivalently $\Tr (1_T)=t$).

For an arbitrary $t\in F$ Deligne defined in \cite{Del07} a symmetric tensor category $\uRep(S_t)$ 
with a distinguished object $X$ which is universal in the following sense:

\begin{proposition} \label{uRepuni}
{\em (\cite[Proposition 8.3]{Del07})} Let $\T$ be a Karoubian symmetric tensor category over $F$.
The functor $\cat{F}\mapsto \cat{F}(X)$ is an equivalence of the category of braided tensor functors 
$\uRep(S_t)\to \T$ with the category of objects $T\in \T$ satisfying {\em (a), (b), (c)} above
and their isomorphisms.
\end{proposition}

Note that for $t=d\in \BZ_{\ge 0}$ Proposition \ref{uRepuni} applied to $T=X_I$ (with $|I|=d$)
produces a canonical functor $\uRep(S_d)\to \Rep(S_d)$ (where $S_d:=S_I$). It is known
(see \cite[Th\'eor\`eme 6.2]{Del07}) that this functor is surjective on Hom's.  Moreover, the morphisms sent to zero by this functor are precisely the so-called negligible morphisms (see \cite[\S 6.1]{Del07}).

\subsection{}
The category $\uRep(S_t)$ is a Karoubian category; it is not abelian for $t=d\in \BZ_{\ge 0}$.
Remarkably, in \cite[Proposition 8.19]{Del07} Deligne defined an {\em abelian} symmetric 
tensor category 
$\uRep^{ab}(S_d)$ and a fully faithful braided tensor functor $\uRep(S_d)\to \uRep^{ab}(S_d)$.\footnote{We refer the reader to \cite[\S 5.8]{Del07} for an example of Karoubian symmetric tensor category
which admits no braided tensor functor to an abelian symmetric tensor category.}  
The main
goal of this paper is to prove a certain universal property of the category $\uRep^{ab}(S_d)$ 
conjectured in \cite[Conjecture 8.21]{Del07}. 

To state this property we need to use the language
of algebraic geometry {\em within} an abelian symmetric tensor category $\T$ (see \cite{Del90}). Namely, for an
object $T\in \T$ satisfying (a), (b), (c) above we can talk about the (affine) $\T-$scheme $\bI :=\Spec(T)$
and the affine group scheme $S_\bI$ of its automorphisms, see \cite[\S 8.10]{Del07}. Furthermore,
assume that the category $\T$ is {\em pre-Tannakian} (see \S \ref{preT} below), 
that is it satisfies finiteness conditions from \cite[2.12.1]{Del90}. Recall that in this case 
a {\em fundamental group} of $\T$ is defined in \cite[\S 8.13]{Del90}.
 This is an affine group scheme $\pi \in \T$ which acts functorially on
any object of $\T$ and this action is compatible with a formation of tensor products. In particular,
the action of $\pi$ on $T$ gives a homomorphism $\eps: \pi \to S_\bI$. Let $\Rep(S_\bI)$ be the
category of representations of $S_\bI$ (see \cite[\S 8.10]{Del07}) and let $\Rep(S_\bI,\eps)$ 
be the full subcategory of $\Rep(S_\bI)$ consisting of such representations $\rho : S_\bI \to GL(V)$
that the action $\rho \circ \eps$ of $\pi$ on $V$ coincides with the canonical action 
(see \cite[\S 8.20]{Del07}).  $\Rep(S_\bI,\eps)$ is an abelian symmetric tensor category and $T$ is one of its
objects. It follows that the functor $\cat{F}: \uRep(S_t)\to \T$ constructed in Proposition \ref{uRepuni}
factorizes as $\uRep(S_t)\xrightarrow{\cat{F}_T} \Rep(S_\bI,\eps)\to \T$ where the functor $\cat{F}_T$ is constructed 
by applying Proposition \ref{uRepuni} to $T\in \Rep(S_\bI,\eps)$ and $\Rep(S_\bI,\eps)\to \T$ is the forgetful
functor. Here is the main result of this paper:

\begin{theorem}\label{main}
{\em (cf. \cite[8.21.2]{Del07})} Let $\T$ be a pre-Tannakian category and $T\in \T$ be an object
satisfying {\em (a), (b), (c)} from \S \ref{Tabc} with $t=d\in \BZ_{\ge 0}\subset F$. Then the category
$\Rep(S_\bI,\eps)$ endowed with the functor $\cat{F}_T: \uRep(S_d)\to \Rep(S_\bI,\eps)$ is equivalent
to one of the following:

{\em (a)} $\Rep(S_d)$ together with the functor $\uRep(S_d)\to \Rep(S_d)$ from \S \ref{Tabc};

{\em (b)} $\uRep^{ab}(S_d)$ together with the fully faithful functor $\uRep(S_d)\to \uRep^{ab}(S_d)$ above.
\end{theorem}

\begin{remark} We note that a similar (and easier) statement holds true for $t\not \in \BZ_{\ge 0}$,
see \cite[Corollary B2]{Del07}.
\end{remark}

\subsection{} 
The forgetful functor $\Rep(S_\bI,\eps)\to \T$ above is an exact braided tensor functor.
Thus Theorem \ref{main} implies that for a pre-Tannakian category $\T$ 
a braided tensor functor $\cat{F}: \uRep(S_d)\to \T$ either factorizes through 
$\uRep(S_d)\to \Rep(S_d)$ or
extends to an exact tensor functor $\uRep^{ab}(S_d)\to \T$. A crucial step in our proof of 
Theorem \ref{main} is a construction
of pre-Tannakian category $\K_d^0$ and fully faithful embedding $\uRep(S_d)\subset \K_d^0$ such that
we have the following {\em extension property}: 
a tensor (not necessarily braided) functor $\uRep(S_d)\to \T$ either factorizes through 
$\uRep(S_d)\to \Rep(S_d)$ or 
extends to an exact tensor functor $\K_d^0\to \T$, see \S \ref{extension}. 
Then we use general properties of the fundamental
groups from \cite[\S 8]{Del90} in order to prove that $\K_d^0$ satisfies the universal property as in 
Theorem \ref{main} and, in fact, is equivalent to $\uRep^{ab}(S_d)$.

The following analogy plays a significant role in the proof of Theorem \ref{main}.
Let $TL(q)$ be the {\em Temperley-Lieb category}, see e.g.~\cite[\S A1]{GW}. Assume that $q$ is a nontrivial root of unity.
It is well known that the
category $TL(q)$ is tensor equivalent to the category of {\em tilting modules} over  quantum $SL(2)$,
see e.g.~\cite[proof of Theorem 2.4]{O}. Thus $TL(q)$ is a Karoubian tensor category (braided but not symmetric) endowed with
a fully faithful functor to the abelian tensor category $\C_q$ of finite dimensional representations of 
quantum $SL(2)$.
On the other hand there exists a well known semisimple tensor category $\bar \C_q$ and a full tensor functor
$TL(q)\twoheadrightarrow \bar \C_q$, see e.g.~\cite[\S 4]{A}. We consider the diagram $\bar \C_q\twoheadleftarrow TL(q)\subset \C_q$
as a counterpart of the diagram $\Rep(S_d)\twoheadleftarrow \uRep(S_d)\subset \uRep^{ab}(S_d)$.

The main technical result of \cite{O} states that tensor functors $TL(q)\to \D$ to certain abelian tensor categories $\D$ factorize
either through $TL(q)\to \bar \C_q$ or through $TL(q)\subset \C_q$ (see \cite[\S 2.6]{O}) 
which is reminiscent of the extension property of the category $\K_d^0$ above, see also \cite[Remark 2.10]{O}.
Thus in the construction of $\K_d^0$ we follow the strategy from \cite{O} with crucial use of information from \cite{CO}. 
Namely, we find $\K_d^0$ inside the homotopy category of $\uRep(S_d)$ as a heart of a suitable {\em $t-$structure} 
(see \S \ref{deft}). The definition of the $t-$structure is based on Lemma \ref{Delemma} (due to P.~Deligne) 
and almost immediately implies the extension property of the category $\K_d^0$ mentioned above. 
However, the verification of the axioms of a $t-$structure is quite nontrivial. 
To do this we use a decomposition of the category $\uRep(S_d)$
into {\em blocks} described in \cite[Theorem 5.3]{CO}. We provide a blockwise description of the $t-$structure above in \S \ref{tblock}. 
We then observe that the description above coincides with the description of a well known $t-$structure on the blocks
of the Temperley-Lieb category.

\subsection{Acknowledgments} This paper owes its existence to Pierre Deligne who explained a proof
of Lemma \ref{Delemma}, which is crucial to this paper, to the second named author when he was visiting Institute for Advanced Study. 
Both authors are happy to express their deep gratitude to him and to the Institute for Advanced Study which made
this interaction possible. The authors are also very grateful to Alexander Kleshchev who initiated this project.
We also thank Michael Finkelberg and Friedrich Knop for their
interest in this work and Darij Grinberg for his detailed comments. The work of the second named author was partially supported by the NSF grant DMS-0602263.  

\section{Preliminaries}
\subsection{Tensor categories terminology} \label{preT}
In this paper a {\em tensor (or monoidal) category} is a category with a tensor product functor endowed with 
an associativity constraint and a unit object $\be$, see e.g.~\cite[Definition 1.1.7]{BK}. Recall that a tensor category
is called {\em rigid} if any object admits both a left and right dual, see \cite[Definition 2.1.1]{BK}.
A {\em braided tensor category} is a tensor category equipped with a braiding, see \cite[Definition 1.2.3]{BK}.
A {\em symmetric tensor category} is a braided tensor category such that the square
of the braiding is the identity.

Recall that $F$ is a fixed field of characteristic zero. All categories and functors considered 
in this paper are going to be $F-$linear. So,
an {\em $F-$linear tensor category} (or {\em tensor category over $F$}) 
is a tensor category which is $F-$linear 
(but not necessarily additive) and such that the tensor product functor is $F-$bilinear. 
A {\em Karoubian tensor category} over $F$ is an $F-$linear tensor category which is Karoubian as an
$F-$linear category (i.e. it is additive and every idempotent endomorphism is a projection
to a direct summand).  A \emph{tensor ideal} $\cat{I}$ in a tensor category $\cat{T}$ consists of subspaces $\cat{I}(X,Y)\subset\Hom_\cat{T}(X,Y)$ for every $X,Y\in\cat{T}$ such that (i) $h\circ g\circ f\in\cat{I}(X, W)$ whenever $f\in\Hom_\cat{T}(X, Y), g\in\cat{I}(Y,Z), h\in\Hom_\cat{T}(Z,W)$, and (ii) $f\otimes\id_Z\in\cat{I}(X\otimes Z, Y\otimes Z)$ whenever $f\in\cat{I}(X, Y)$.  For example, if the category $\cat{T}$ has a well defined trace the collection 
of  \emph{negligible morphisms}\footnote{
Recall that a morphism $f\in \Hom_\cat{T}(X,Y)$ is negligible if $\Tr(fg)=0$ for any $g\in \Hom_\cat{T}(Y,X)$.  We will call an object negligible if its identity morphism is negligible.} forms a tensor ideal, see \cite[\S A1.3]{GW}.

Finally we say that an $F-$linear symmetric tensor category $\T$ is {\em pre-Tannakian} 
if the following conditions are satisfied:

(a) all Hom's are finite dimensional vector spaces over $F$ and $\End(\be)=F$;

(b) $\T$ is an abelian category and all objects have finite length;

(c) $\T$ is rigid.

\begin{remark} In the terminology of \cite{Del90} a pre-Tannakian category is the same as a
``cat\'egorie tensorielle" (see \cite[\S 2.1]{Del90}) satisfying a finiteness assumption
\cite[2.12.1]{Del90}. This is precisely the class of tensor categories over $F$ for which a {\em fundamental
group} (see \cite[\S 8]{Del90}) is defined.
\end{remark}

\subsection{The category $\uRep(S_t)$} We recall here briefly the construction of the category 
$\uRep(S_t)$ following \cite[\S 2]{CO}. We refer the reader to {\em loc. cit.} and \cite[\S 8]{Del07}
for much more detailed exposition.

\subsubsection{The category $\uRep_0(S_t)$} Let $A$ be a finite set. A {\em partition} $\pi$ of $A$
is a collection of nonempty subsets $\pi_i\subset A$ such that $A=\sqcup_i\pi_i$ (disjoint union); the subsets
$\pi_i$ are called {\em parts} of the partition $\pi$. We say that partition $\pi$ is {\em finer} than
partition $\mu$ of the same set if any part of $\pi$ is a subset of some part of $\mu$.
For three finite sets $A, B, C$ and the partitions
$\pi$ of $A\sqcup B$ and $\mu$ of $B\sqcup C$ we define the partition $\mu \star \pi$ of
$A\sqcup B\sqcup C$ as the finest partition such that parts of $\pi$ and $\mu$ are subsets of
its parts.  The partition $\mu \star \pi$ induces a partition $\mu \cdot \pi$ of $A\sqcup C$
such that parts of $\mu \cdot \pi$ are nonempty intersections of parts of $\mu \star \pi$ with
$A\sqcup C\subset A\sqcup B\sqcup C$; we also define an integer $\ell(\mu,\pi)$ which is 
the number of parts of $\mu \star \pi$ contained in $B$.

\begin{definition} Given $t\in F$, we define the $F-$linear symmetric tensor category $\uRep_0(S_t)$ as follows:

Objects: finite sets; object corresponding to a finite set $A$ is denoted $[A]$.

Morphisms: $\Hom([A],[B])$ is the $F-$linear span of partitions of $A\sqcup B$; composition
of morphisms represented by partitions $\pi \in \Hom([A],[B])$ and $\mu \in \Hom([B],[C])$ 
is $t^{\ell(\mu,\pi)}\mu \cdot \pi \in \Hom([A],[C])$.

Tensor product: disjoint union (see \cite[Definition 2.15]{CO}); unit object is $[\varnothing]$; 
tensor product of morphisms, associativity and commutativity
constraints are the obvious ones (see \cite[\S2.2]{CO}).
\end{definition} 

The category $\uRep_0(S_t)$ has a distinguished object $[pt]$ where $pt$ is a
one-element set. The object $[pt]$ has a natural structure of commutative associative algebra
in $\uRep_0(S_t)$ where the multiplication (resp. unit) map is given  by the partition of $pt \sqcup pt \sqcup pt$ (resp. $pt$)
consisting of one part. It is immediate to check that the object $[pt]$ satisfies conditions (a), (b), (c)
from \S\ref{Tabc}. Moreover, we have the following universal property:

\begin{proposition} \label{uni0}
 Let $\T$ be an $F-$linear symmetric tensor category. The functor from the category of braided tensor functors 
$\cat{F}: \uRep_0(S_t)\to \T$
to the category of objects $T\in \T$ satisfying (a), (b), (c) from \S \ref{Tabc} and their isomorphisms, which sends  $\cat{F}\mapsto \cat{F}([pt])$ and sends natural transformations $(\eta:\cat{F}\to\cat{F'})\mapsto\eta_{[pt]}$ is an equivalence of categories.
\end{proposition}

\noindent\emph{Sketch of proof.} We restrict ourselves by a description of the inverse functor on objects; for more details
see \cite[\S 8]{Del07}. So assume that $T\in \T$ satisfies (a), (b), (c) from \ref{Tabc}. 
We define $\cat{F}([A])=T^{\ot A}$ (here $T^{\ot A}$ is a tensor product of copies of $T$ labeled
by elements of $A$; since the category $\T$ is symmetric this is well defined). The tensor structure
on the functor $\cat{F}$ will be given by the obvious isomorphisms $T^{\ot A\sqcup B}=T^{\ot A}\ot
T^{\ot B}$. It remains to define $\cat{F}$ on the morphisms. Observe that a morphism
from $\Hom([A],[B])$ represented by a partition $\pi$ of $A\sqcup B$ is a tensor product
of morphisms corresponding to partitions with precisely one part $\pi =\ot_i \pi_i$. Thus it
is sufficient to define $\cat{F}(\pi)$ only for $\pi$ consisting of one part $A\sqcup B$. In this
case we set $\cat{F}(\pi)=T^{\ot A}\to T\to T^{\ot B}$ where the first map is the multiplication 
morphism $T^{\ot A}\to T$ and the second one is the dual to the multiplication morphism 
$T^{\ot B}\to T$ where $T$ and $T^*$ are identified via (b) from \S \ref{Tabc}. One verifies
that the assumptions (a), (b), (c) from \S \ref{Tabc} ensure that the tensor functor $\cat{F}$ is well
defined. \hfill$\square$

\subsubsection{The categories $\uRep(S_t)$ and $\uRep^{ab}(S_d)$}\label{222}
\begin{definition} (cf. \cite[D\'efinition 2.17]{Del07} or \cite[Definition 2.19]{CO})
The category $\uRep(S_t)$ is the Karoubian (or pseudo-abelian) envelope\footnote{we refer
the reader to \cite[\S 1.7-1.8]{Del07} for the discussion of this notion.}
 of the category $\uRep_0(S_t)$. 
\end{definition}

It follows immediately from Proposition \ref{uni0} that the
category $\uRep(S_t)$ has universal property from Proposition \ref{uRepuni}. We now use
this universal property to construct Deligne's category $\uRep^{ab}(S_d)$ from
the introduction.

It is known (see \cite[Th\'eor\`eme 2.18]{Del07} or \cite[Corollary 5.21]{CO}) 
that the category $\uRep(S_t)$ is semisimple
(and hence pre-Tannakian) for $t\not \in \BZ_{\ge 0}$. In particular, the category
$\uRep(S_{-1})$ is pre-Tannakian, so its fundamental group $\pi$ is defined.
For any $d\in \BZ_{\ge 0}$ we can consider the
commutative associative algebra with non-degenerate trace pairing $T_d\in \uRep(S_{-1})$ which is a direct sum of $[pt]$ and
$d+1$ copies of the algebra $\be=[\varnothing]$. Clearly, $\dim(T_d)=d$, so we can use Proposition
\ref{uRepuni} to construct a symmetric tensor functor $\uRep(S_d)\to \uRep(S_{-1})$. 
Using the general properties of the fundamental group we get a factorization
of this functor as $\uRep(S_d)\to \Rep(S_\bI,\eps)\to \uRep(S_{-1})$ (here $\bI =\Spec(T_d)$
and $\eps :\pi \to S_\bI$ is the canonical homomorphism). It is clear that the category
$\Rep(S_\bI,\eps)$ is pre-Tannakian; it is proved in \cite[Proposition 8.19]{Del07}
that the functor $\uRep(S_d)\to \Rep(S_\bI,\eps)$ is fully faithful. We set $\uRep^{ab}(S_d):=
\Rep(S_\bI,\eps)$; as explained above this is a pre-Tannakian category and we have 
a fully faithful braided tensor functor $\uRep(S_d)\to \uRep^{ab}(S_d)$.

\begin{remark}\label{tensornonzero} The existence of the embedding $\uRep(S_t)\subset \uRep^{ab}(S_t)$ implies
that $Y_1\ot Y_2\ne 0$ for nonzero objects $Y_1, Y_2\in \uRep(S_t)$ (this is true in
any abelian rigid tensor category with simple unit object). The same result can be proved directly as follows.
Given finite sets $A$ and $B$, it follows from the definition of tensor products that the obvious map 
$\End([A])\ot \End([B])\to \End([A]\ot [B])=\End([A\sqcup B])$ is injective.  Since any indecomposable object of $\uRep(S_t)$ is the image of a primitive idempotent 
$e\in \End([A])$ for some finite set $A$, see e.g.~\cite[Proposition 2.20]{CO}, it follows that the tensor product of two nonzero morphisms in $\uRep(S_t)$ is nonzero.  The statement for objects follows by considering their identity morphisms.
\end{remark}

\subsubsection{Indecomposable objects of the category $\uRep(S_t)$}
The indecomposable objects of the category $\uRep(S_t)$ are classified up to
isomorphism in \cite[Theorem 3.3]{CO}.
The isomorphism classes are labeled by the Young diagrams of all sizes in the following
way. Let $\lambda$ be a Young diagram of size $n=|\lambda|$ and let 
$y_\lambda$ be the corresponding primitive idempotent in $FS_n$, the group algebra of the symmetric group\footnote{Here $y_\lambda$ is a scalar multiple of the so-called Young symmetrizer (see for instance \cite{FH}).}.  The symmetric braiding gives rise to an action of $S_n$ on $[pt]^{\otimes n}$; let $[pt]^\lambda$ denote the image of $y_\lambda\in\End([pt]^{\otimes n})$.
%
%
%
For any Young diagram $\lambda$ of size $|\lambda|$ there is a unique indecomposable object $L(\lambda)\in \uRep(S_t)$ characterized by the following properties:

(a) $L(\lambda)$ is not a direct summand of $[pt]^{\ot k}$ for $k<|\lambda|$;

(b) $L(\lambda)$ is a direct summand (with multiplicity 1) of $[pt]^\lambda$.

It is proved in \cite[Theorem 3.3]{CO} that the indecomposable objects $L(\lambda)$ 
are well defined up to isomorphism, and any indecomposable object of $\uRep(S_t)$ is isomorphic to
precisely one $L(\lambda)$.

\subsubsection{Blocks of the category $\uRep(S_t)$}\label{StBlocks}
Let $\A$ be a Karoubian category such that any object decomposes into a finite direct sum
of indecomposable objects. The set of isomorphism classes of indecomposable objects
of $\A$ splits into {\em blocks} which are equivalence classes of the weakest equivalence 
relation for which two indecomposable objects are equivalent whenever there exists a 
nonzero morphism between them. We will also use the term block to refer to a full
subcategory of $\A$ generated by the indecomposable objects in a single block.

The main result of \cite{CO} is the description of blocks of the category $\uRep(S_t)$. We 
describe the results of {\em loc. cit.} here. We will represent a Young diagram $\lambda$ as 
an infinite non-increasing sequence $(\lambda_1,\lambda_2,\ldots)$ of nonnegative integers such that
$\lambda_k=0$ for some $k>0$, see \cite[\S 1.1]{CO}. For a Young diagram $\lambda$ and $t\in F$
we define a sequence $\mu_\lambda (t)=(t-|\lambda|, \lambda_1-1, \lambda_2-2,\ldots)$.

\begin{theorem} \label{CObl} {\em (\cite[Theorem 5.3]{CO})} 
The objects $L(\lambda)$ and $L(\lambda')$ of $\uRep(S_t)$
are in the same block if and only if $\mu_\lambda(t)$ is a permutation of $\mu_{\lambda'}(t)$.
\end{theorem}

Let $\B$ be the set of blocks of the category $\uRep(S_t)$; for any $\mathsf{b}\in \B$ let us denote
by $\uRep_{\mathsf{b}}(S_t)$ the corresponding subcategory of $\uRep(S_t)$; we have a decomposition 
$\uRep(S_t)=\oplus_{\mathsf{b}\in \B}\uRep_{\mathsf{b}}(S_t)$. 

\begin{proposition}\label{StBlockProp} Let $\mathsf{b}\in \B$. One of the following holds:

(i) $\mathsf{b}$ is semisimple (or trivial): the category $\uRep_{\mathsf{b}}(S_t)$ is equivalent to the category 
$\Vec_F$ of finite dimensional $F-$vector spaces as an additive category. We will denote
by $L=L(\mathsf{b})$ the unique indecomposable object of this block. Then  $\dim(L)=0$, or, equivalently, 
$\id_L$ is negligible.

(ii) $\mathsf{b}$ is non-semisimple (or infinite). In this case the additive category $\uRep_{\mathsf{b}}(S_t)$ is described
in \cite[\S6]{CO} (in particular, it does not depend on a choice of non-semisimple block $\mathsf{b}$). There is
a natural labeling of indecomposable objects of the category $\uRep_{\mathsf{b}}(S_t)$ by nonnegative
integers; we will denote these objects by $L_0, L_1, \ldots$.  Then $\dim(L_i)=0$ for $i>0$
and $\dim(L_0)\ne 0$, that is  $\id_{L_i}$ is negligible if and only if $i>0$.
\end{proposition}

Further, it is shown in \cite{CO} that for any $t\in F$ there are infinitely many semisimple blocks
and finitely many (precisely the number of Young diagrams of size $t$) non-semisimple blocks.  In particular,  for $t\not \in \BZ_{\ge 0}$ all blocks are
semisimple (hence the category $\uRep(S_t)$ is semisimple). 

\subsection{Temperley-Lieb category} \label{TLcat}
The results on the category $\uRep(S_t)$ in many respects
are parallel to the results on the Temperley-Lieb category $TL(q)$. We recall the definition
and some properties of this category here.

\begin{definition} \label{TLdef} (see e.g.~\cite[\S A1.2]{GW})
Let $q$ be a nonzero element of an algebraic closure of $F$ such that $q+q^{-1}\in F$.
We define the $F-$linear tensor category $TL_0(q)$ as follows:

Objects: finite subsets of $\BR$ considered up to isotopy; we will denote the object corresponding
to the set $A$ by $\langle A\rangle$.

Morphisms: $\Hom(\langle A\rangle ,\langle B\rangle )$ is the $F-$linear span of 
 one dimensional submanifolds of $\BR \times [0,1]$ 
with boundary $A\sqcup B$ where $A\subset \BR \times 0$ and $B\subset \BR \times 1$
(such submanifolds are called embedded unoriented {\em bordisms} from $A$ to $B$) 
modulo the relation $[bordism \sqcup circle]=(q+q^{-1})[bordism]$;  composition is given
by juxtaposition.  

Tensor product: disjoint union (write $\BR =\BR_{<0}\sqcup 0 \sqcup \BR_{>0}$ and identify
$\BR_{<0}$ and $\BR_{>0}$ with $\BR$); the unit object is $\langle \varnothing\rangle$; 
tensor product of morphisms and associativity constraint are the obvious ones.
\end{definition}

Next we define the category $TL(q)$ as the Karoubian envelope of the category $TL_0(q)$.
The category $TL(q)$ has a universal property (see e.g.~\cite[Theorem 2.4]{O}) but we don't need it here.
The indecomposable objects of the category $TL(q)$ are labeled by nonnegative integers:
for any $i\in \BZ_{\ge 0}$ there is a unique indecomposable object $V_i$ which is
a direct summand (with multiplicity 1) of $\langle pt\rangle^{\otimes i}$ but is not a direct summand of 
$\langle pt\rangle^{\otimes k}$ whenever $k<i$.

The category $TL(q)$ is semisimple for generic values of $q$; more precisely the category
$TL(q)$ is not semisimple precisely when exists a positive integer $l$ such that
$1+q^2+\ldots +q^{2l}=0$ (we will denote the smallest such integer by $l_q$). Assume
that the category $TL(q)$ is not semisimple. Then we have a full tensor functor $TL(q)\to \bar \C_q$
and a fully faithful tensor functor $TL(q)\to \C_q$ where $\bar \C_q$ is a semisimple tensor category
(sometimes called the ``Verlinde category") and $\C_q$ is the abelian tensor category of 
finite dimensional representations of quantum $SL(2)$, see e.g.~\cite[Theorem 2.4]{O}.

The blocks of the category $TL(q)$ are well known. Similarly to the case of the category $\uRep(S_d)$
there are infinitely many semisimple blocks (which are equivalent to the category $\Vec_F$ as an 
additive category) and finitely many (precisely $l_q$) non-semisimple blocks. The following
observation is very important for this paper:

\begin{proposition}\label{TLblocks} {\em (\cite[Remark 6.5]{CO})} All non-semisimple blocks of the category $TL(q)$ are equivalent
as additive categories. Moreover, they are equivalent to the category $\uRep_{\mathsf{b}}(S_d)$ where
$\mathsf{b}$ is any non-semisimple block of the category $\uRep(S_d)$.
\end{proposition}

\begin{remark} \label{TLlabel}
We can transport a labeling of indecomposable objects of $\uRep_{\mathsf{b}}(S_d)$ (see Proposition \ref{StBlockProp} (ii))
to a non-semisimple block
of the category $TL(q)$ via the equivalence of Proposition \ref{TLblocks} (it is easy to see that the resulting labeling does not
depend on a choice of the equivalence).
\end{remark}

Recall that the category $TL(q)$ has a natural spherical structure and so the dimensions $\dim_{TL(q)}(Y)$ 
of objects $Y\in TL(q)$ are defined, see e.g.~\cite[\S A1.3]{GW}.
The following result is well known, see e.g.~\cite[(1.6) and Proposition 3.5]{A}:

\begin{lemma} \label{TLneglig}
Let $L$ be a unique indecomposable object in a semisimple block of $TL(q)$. Then $\dim_{TL(q)}(L)=0$.
For a non-semisimple block we have $\dim_{TL(q)}(L_i)=0$ for $i>0$ and
$\dim_{TL(q)}(L_0)\ne 0$ where $L_i$ are indecomposable objects in this block labeled as in Remark \ref{TLlabel} $\square$
\end{lemma}

\section{Tensor ideals and the object $\Delta \in \uRep(S_d)$}

In this section we define objects $\Delta_n\in\uRep(S_t)$ for $n\in\Z_{\geq0}$ and $t\in F$.  We then give $\Delta_n$ the structure of a commutative associative algebra in $\uRep(S_t)$ and study many $\Delta_n$-modules.  Finally, using our results on the objects $\Delta_n$, we classify tensor ideals in $\uRep(S_d)$ when $d$ is a nonnegative integer.  Before defining the objects $\Delta_n$ we prove the following easy observation which will be used throughout this section.
\begin{proposition}\label{polyt} Suppose $A_0,\ldots, A_n$ and $B_0,\ldots, B_m$ are finite sets with $A_0=B_0$ and $A_n=B_m$.   Suppose further that $f_i$ (resp. $g_i$) is an $F$-linear combination of partitions of $A_{i-1}\sqcup A_i$ (resp. $B_{i-1}\sqcup B_i$) whose coefficients do not depend on $t$ for all $1\leq i\leq n$ (resp. $1\leq i\leq m$).  If $f_n\cdots f_1=g_m\cdots g_1$ in $\uRep_0(S_t)$ for infinitely many values of $t\in F$, then $f_n\cdots f_1=g_m\cdots g_1$ in $\uRep_0(S_t)$ for all $t\in F$.
\end{proposition}
\begin{proof} For each $t\in F$ and partition $\pi$ of $A_0\sqcup A_n=B_0\sqcup B_m$, let $a_\pi(t)\in F$ (resp. $b_\pi(t)\in F$) be such that 
$f_n\cdots f_1=\sum_\pi a_\pi(t) \pi$ (resp. $g_m\cdots g_1=\sum_\pi b_\pi(t)\pi$) in $\uRep_0(S_t)$  where the sum is taken over all partitions $\pi$ of $A_0\sqcup A_n=B_0\sqcup B_m$.  Then $f_n\cdots f_1=g_m\cdots g_1$ in $\uRep_0(S_t)$ if and only if $a_\pi(t)=b_\pi(t)$ for all $\pi$.  By the definition of composition in $\uRep_0(S_t)$, both $a_\pi(t)$ and $b_\pi(t)$ are polynomials in $t$ for each $\pi$.  The result follows since a polynomial in $t$ is determined by finitely many values of $t$.
\end{proof}

\subsection{The objects $\Delta_n\in \uRep(S_t)$}

Suppose $n$ is a nonnegative integer and let $A_n=\{i~|~1\leq i\leq n\}$.  Consider the endomorphism $x_n=x_{\id_n}:[A_n]\to[A_n]$ in $\uRep_0(S_t)$ (see \cite[Equation (2.1)]{CO}).  

\begin{proposition}\label{xidemp} $x_n$ is an idempotent which is equal to its dual for all $n\geq 0$.
\end{proposition}

\begin{proof} The fact that $x_n^\ast=x_n$ follows from the definition of $x_n$.  By Proposition \ref{polyt}, it suffices to show $x_n$ is an idempotent in $\uRep_0(S_t)$ for infinitely many values of $t$. 
It follows from \cite[Theorem 2.6 and Equation (2.2)]{CO} that $x_n$ is an idempotent in $\uRep_0(S_t)$ whenever $t$ is an integer greater than $2n$. 
\end{proof}

Since $\uRep(S_t)$ is a Karoubian category (i.e. $\uRep(S_t)$ contains images of idempotents) the following definition is valid.

\begin{definition} Let $\Delta_n\in\uRep(S_t)$ denote the image of the idempotent $x_n$.\footnote{In the notation of \cite{CO}, $\Delta_n=([n], x_n)$.} 
\end{definition}






Note that the commutative associative algebra structure on $[pt]$ extends in an obvious way to a commutative associative algebra structure on $[A_n]\cong[pt]^{\otimes n}$.  Let $\mu_n:[A_n]\otimes[A_n]\to[A_n]$ and $1_n:{\bf 1}\to[A_n]$ denote the multiplication and unit maps respectively.  

\begin{proposition}\label{deltalg} The multiplication map $x_n\mu_n(x_n\otimes x_n):\Delta_n\otimes\Delta_n\to\Delta_n$  gives $\Delta_n$ the structure of a commutative associative algebra in $\uRep(S_t)$ with unit given by $x_n1_n:{\bf 1}\to\Delta_n$.
\end{proposition}

\begin{proof}  
We are required to show the following equalities hold in $\uRep_0(S_t)$:
\begin{equation} \label{D1}
x_n\mu_n(x_n\mu_n(x_n\otimes x_n)\otimes x_n)=x_n\mu_n(x_n\otimes(x_n\mu_n(x_n\otimes x_n)),
 \end{equation}
\begin{equation}\label{D2}
x_n\mu_n(x_n1_n\otimes x_n)=x_n=x_n\mu_n(x_n\otimes x_n1_n),
\end{equation}
\begin{equation}\label{D3}
x_n\mu_n(x_n\otimes x_n)\beta_{n,n}(x_n\otimes x_n)=x_n\mu_n(x_n\otimes x_n),
\end{equation} where $\beta_{n,n}:A_n\otimes A_n\to A_n\otimes A_n$ is the braiding morphism (see for example  \cite[\S2.2]{CO}).  By Proposition \ref{polyt}, it suffices to show (\ref{D1}), (\ref{D2}), (\ref{D3}) hold for infinitely many values of $t$.\footnote{In fact, (\ref{D1}), (\ref{D2}), (\ref{D3}) do not  depend on $t$, so we only need to verify they hold for some $t$. }  Using \cite[Theorem 2.6 and Equation (2.2)]{CO} it is easy to show (\ref{D1}), (\ref{D2}), (\ref{D3}) hold whenever $t$ is a sufficiently large integer.
\end{proof}

By Proposition \ref{deltalg} we can consider the category $\Delta_n$-mod of all left $\Delta_n$-modules.  

\subsection{Some $\Delta_n$-modules}
 Suppose $j$ is a nonnegative integer with $1\leq j\leq n$.  Give a finite set $X$, let $\Theta_X^j:\Hom_{\uRep(S_t)}( A_n, X)\to\Hom_{\uRep(S_t)}(A_{n+1}, X)$ and $\Theta^X_j:\Hom_{\uRep(S_t)}(X, A_n)\to\Hom_{\uRep(S_t)}(X, A_{n+1})$ be the $F$-linear maps defined on partitions as follows:  if $\pi$ is a partition of $X\sqcup A_n$, then $\Theta_X^j(\pi)=\Theta^X_j(\pi)$ is the unique partition of $X\sqcup A_{n+1}$ which restricts to $\pi$ and has $j$ and $n+1$ in the same part.  Now let $\Theta_j:\End_{\uRep(S_t)}(A_n)\to\End_{\uRep(S_t)}(A_{n+1})$ be the $F$-linear map $\Theta_j=\Theta_{A_n}^j\circ\Theta^{A_n}_j$.  It is easy to check that $\Theta_j$ is an injective (non-unital) $F$-algebra homomorphism for each $1\leq j\leq n$.  In particular, by Proposition \ref{xidemp}, $x_{n,j}:=\Theta_j(x_n)$ is an idempotent for each $j$.  
 
\begin{definition} 
Let $\Delta_n(j)\in\uRep(S_t)$ denote the image of $x_{n, j}$. 
\end{definition}

Next we give $\Delta_n(j)$ the structure of a $\Delta_n$-module.  
Let $\alpha=x_{n,j}\Theta^{A_n}_j(x_n)x_n:\Delta_n\to\Delta_n(j)$ and $\beta=x_{n,j}\Theta_j(\mu_n)(x_{n,j}\otimes x_{n,j}):\Delta_n(j)\otimes\Delta_n(j)\to\Delta_n(j)$.  Finally, let $\phi=\beta(\alpha\otimes x_{n,j}):\Delta_n\otimes\Delta_n(j)\to\Delta_n(j)$.
 
\begin{proposition}\label{Deltaj}

(1)  The map $\phi$  gives $\Delta_n(j)$ the structure of a $\Delta_n$-module.

(2) The map $ x_{n,j}\Theta_j^{A_n}(\id_{A_n})x_n:\Delta_n\to\Delta_n(j)$ is an isomorphism of $\Delta_n$-modules with inverse $x_n\Theta^j_{A_n}(\id_{A_n}) x_{n,j}$.

\end{proposition}

\begin{proof} 
For part (1) we are required to show the following equation holds in $\uRep_0(S_t)$:
\begin{equation}\label{Dj1} 
\begin{array}{l}
x_{n,j}\Theta_j(\mu_n)(( x_{n,j}\Theta^{A_n}_j(x_n)x_n\mu_n(x_n\otimes x_n))\otimes x_{n,j})\\
= x_{n,j}\Theta_j(\mu_n)( x_{n,j}\Theta^{A_n}_j(x_n)x_n\otimes( x_{n,j}\Theta_j(\mu_n)(( x_{n,j}\Theta^{A_n}_j(x_n)x_n)\otimes x_{n,j}))).
 \end{array}
\end{equation}
For part (2) we are required to show the following equations hold in $\uRep_0(S_t)$:
\begin{equation}\label{Dj2}
 x_{n,j}\Theta_j^{A_n}(\id_{A_n})x_n\Theta^j_{A_n}(\id_{A_n}) x_{n,j}=x_{n,j},
\end{equation}
\begin{equation}\label{Dj3}
x_n\Theta^j_{A_n}(\id_{A_n})  x_{n,j}\Theta_j^{A_n}(\id_{A_n})x_{n}=x_{n}.
\end{equation}  
Now use Proposition \ref{polyt} and \cite[Theorem 2.6 and Equation (2.2)]{CO}.
\end{proof}

Next, we give the object $\Delta_{n+1}$ the structure of a $\Delta_n$-module.  To do so, set $\psi=x_{n+1}(\mu_n\otimes\id_{[pt]})(x_n\otimes x_{n+1}):\Delta_n\otimes\Delta_{n+1}\to\Delta_{n+1}$.

\begin{proposition} The map $\psi$  gives $\Delta_{n+1}$ the structure of a $\Delta_n$-module.
\end{proposition}

\begin{proof} 
We are required to show the following equation holds in $\uRep_0(S_t)$:
\begin{equation}\label{Dplus1}\begin{array}{l}
x_{n+1}(\mu_n\otimes\id_{[pt]})((x_n\mu_n(x_n\otimes x_n))\otimes x_{n+1})\\
=x_{n+1}(\mu_n\otimes\id_{[pt]})(x_n\otimes(x_{n+1}(\mu_n\otimes\id_{[pt]})(x_n\otimes x_{n+1})).
\end{array}
\end{equation}  Now use Proposition \ref{polyt} and \cite[Theorem 2.6 and Equation (2.6)]{CO}.
\end{proof}

The following lemma 
will be important for us later.

\begin{lemma}\label{Deltapt} $\Delta_n\otimes[pt]\cong\Delta_{n+1}\oplus\Delta_n(1)\oplus\cdots\oplus\Delta_n(n)$ in the category $\Delta_n$-mod.
\end{lemma}

\begin{proof}  
First, using Proposition \ref{polyt} and \cite[Theorem 2.6 and Equation (2.6)]{CO} it is easy to show that the following identities hold in $\uRep_0(S_t)$:
\begin{equation}\label{ortho}
\begin{array}{ll}
x_n\otimes\id_{[pt]}=x_{n+1}+\sum\limits_{1\leq j\leq n}x_{n, j},\\
x_{n,j}x_{n+1}=0=x_{n+1}x_{n,j} & (1\leq j\leq n),\\
x_{n,j}x_{n,k}=\delta_{j,k}x_{n,j} & (1\leq j,k\leq n).
\end{array}
\end{equation}
Next, define $\Psi:\Delta_n\otimes[pt]\to\Delta_{n+1}\oplus\Delta_n(1)\oplus\cdots\oplus\Delta_n(n)$ by $$\Psi=\left[
\begin{array}{c}
x_{n+1}(x_n\otimes \id_{[pt]})\\
x_{n,1}(x_n\otimes \id_{[pt]})\\
\vdots\\
x_{n,n}(x_n\otimes \id_{[pt]})
\end{array}
\right].$$
Using (\ref{ortho}) is is easy to check that $\Psi$ is an isomorphism in $\uRep(S_t)$ with inverse $$\Psi^{-1}=\left[
\begin{array}{cccc}
(x_n\otimes \id_{[pt]})x_{n+1} &
(x_n\otimes \id_{[pt]})x_{n,1}&
\cdots&
(x_n\otimes \id_{[pt]})x_{n,n}
\end{array}
\right].$$
It remains to show that $\Psi$ and $\Psi^{-1}$ are are morphisms in the category $\Delta_n$-mod.  Showing $\Psi$ is a morphism in $\Delta_n$-mod amounts to showing the following equations hold in $\uRep_0(S_t)$:
\begin{equation}\label{Psi}
\begin{array}{l}
x_{n+1}(\mu_n\otimes\id_{[pt]})(x_n\otimes(x_{n+1}(x_n\otimes\id_{[pt]})))=x_{n+1}((x_n\mu_n(x_n\otimes x_n))\otimes\id_{[pt]}),\\
x_{n,j}\Theta_j(\mu_n)((x_{n,j}\Theta^{A_n}_j(x_n)x_n)\otimes(x_{n,j}(x_n\otimes\id_{[pt]})))\\
~\hspace{1.65in}=x_{n,j}((x_n\mu_n(x_n\otimes x_n))\otimes\id_{[pt]})\qquad(1\leq j\leq n).
\end{array}
\end{equation}
To show the equations in (\ref{Psi}) hold, use Proposition \ref{polyt} and \cite[Theorem 2.6 and Equation (2.6)]{CO}.  The proof for $\Psi^{-1}$ is similar.  
\end{proof}

\subsection{The category $\uRep^{\Delta_n}(S_t)$} Let $\Delta_n$-mod$_0$ denote the full subcategory of $\Delta_n$-mod such that a $\Delta_n$-module $M$ is in $\Delta_n$-mod$_0$ if and only if $M\cong\Delta_n\otimes Y$ in $\Delta_n$-mod for some $Y\in\uRep(S_t)$.  Let $\uRep^{\Delta_n}(S_t)$ denote the Karoubian envelope of $\Delta_n\text{-mod}_0$.  The advantage of working in $\uRep^{\Delta_n}(S_t)$ rather than in the category $\Delta_n$-mod is that we can give $\uRep^{\Delta_n}(S_t)$ the structure of a tensor category with relative ease.  Indeed, given $M, M'\in\Delta_n$-mod$_0$ we know $M\cong \Delta_n\otimes Y$ and $M'\cong\Delta_n\otimes Y'$ as $\Delta_n$-modules for some $Y, Y'\in\uRep(S_t)$.  Set $M\otimes_{\Delta_n}M':=\Delta_n\otimes Y\otimes Y'$.  Given $N, N'\in\Delta_n$-mod$_0$ with $N\cong \Delta_n\otimes Z$ and $N'\cong \Delta_n\otimes Z'$ and morphisms $f\in\Hom_{\Delta_n\text{-mod}_0}(M, N)$ and $g\in\Hom_{\Delta_n\text{-mod}_0}(M', N')$, write $\tilde{f}:\Delta_n\otimes Y\arup{\cong} M\arup{f}N\arup{\cong}\Delta_n\otimes Z$ and $\tilde{g}:\Delta_n\otimes Y'\arup{\cong} M'\arup{g}N'\arup{\cong}\Delta_n\otimes Z'$.  Define $f\otimes_{\Delta_n}g:M\otimes_{\Delta_n} M'\to N\otimes_{\Delta_n} N'$ to be the composition $M\otimes_{\Delta_n} M'=\Delta_n\otimes Y\otimes Y'\arup{\tilde{f}\otimes\id_{Y'}}\Delta_n\otimes Z\otimes Y'\arup{\sim}\Delta_n\otimes Y'\otimes Z\arup{\tilde{g}\otimes\id_Z}\Delta_n\otimes Z'\otimes Z\arup{\sim}\Delta_n\otimes Z\otimes Z'=N\otimes_{\Delta_n}N'$.  It is easy to check that $\otimes_{\Delta_n}:\Delta_n\text{-mod}_0\times\Delta_n\text{-mod}_0\to\Delta_n\text{-mod}_0$ is a bifunctor which (with the obvious choice of constraints) makes $\Delta_n$-mod$_0$ into a rigid symmetric tensor category.  The tensor structure on $\Delta_n$-mod$_0$ extends in an obvious way to make $\uRep^{\Delta_n}(S_t)$ a rigid symmetric tensor category too.  

Notice that $\Delta_{n+1}$ is an object in $\uRep^{\Delta_n}(S_t)$.  Indeed,
by Lemma \ref{Deltapt}, the $\Delta_n$-module $\Delta_{n+1}$ is the image of an idempotent of the form $\Delta_n\otimes[pt]\to\Delta_n\otimes[pt]$.  This idempotent is an element of $\End_{\Delta_n\text{-mod}_0}(\Delta_n\otimes[pt])$; hence its image is an object in the Karoubian category $\uRep^{\Delta_n}(S_t)$.   The next two propositions concern the structure of $\Delta_{n+1}\in\uRep^{\Delta_n}(S_t)$.  We start by computing its dimension:

\begin{proposition}\label{Deltadim} $\dim_{\uRep^{\Delta_n}(S_t)}(\Delta_{n+1})=t-n$.
\end{proposition}

\begin{proof} First, by Lemma \ref{Deltapt} and Proposition \ref{Deltaj}(2), $$\dim_{\uRep^{\Delta_n}(S_t)}(\Delta_{n+1})=\dim_{\uRep^{\Delta_n}(S_t)}(\Delta_n\otimes[pt])-n\dim_{\uRep^{\Delta_n}(S_t)}(\Delta_n).$$  Now, consider the tensor functor $\Delta_n\otimes -:\uRep(S_t)\to\uRep^{\Delta_n}(S_t)$.  Since tensor functors preserve dimension, $\dim_{\uRep^{\Delta_n}(S_t)}(\Delta_n)=\dim_{\uRep(S_t)}([\varnothing])=1$ and $\dim_{\uRep^{\Delta_n}(S_t)}(\Delta_n\otimes[pt])=\dim_{\uRep(S_t)}([pt])=t$.
\end{proof}

Our next aim is to show $\Delta_{n+1}\in\uRep^{\Delta_n}(S_t)$ satisfies (a) and (b) from \S\ref{Tabc}.  To do so, let $inc:\Delta_{n+1}\to\Delta_n\otimes[pt]$ and $proj:\Delta_n\otimes[pt]\to\Delta_{n+1}$ denote the morphisms in $\uRep^{\Delta_n}(S_t)$ determined by Lemma \ref{Deltapt}.  Moreover, let $m:(\Delta_n\otimes[pt])\otimes_{\Delta_n}(\Delta_n\otimes[pt])\to\Delta_n\otimes[pt]$ denote the morphism $\Delta_n\otimes[pt]\otimes[pt]\arup{\id_{\Delta_n}\otimes\mu_1}\Delta_n\otimes[pt]$.  Now consider the following morphisms:
\begin{equation}\label{multD}
\Delta_{n+1}\otimes_{\Delta_n}\Delta_{n+1}\arup{inc\otimes_{\Delta_n}inc}(\Delta_n\otimes[pt])\otimes_{\Delta_n}(\Delta_n\otimes[pt])\arup{m}\Delta_n\otimes[pt]\arup{proj}\Delta_{n+1},
\end{equation}
\begin{equation}\label{unitD}
\Delta_n\arup{\id_{\Delta_n}\otimes 1_1}\Delta_n\otimes[pt]\arup{proj}\Delta_{n+1}.
\end{equation}

\begin{proposition}\label{abDelta} With the multiplication and unit maps given by (\ref{multD}) and (\ref{unitD}) respectively, $\Delta_{n+1}\in\uRep^{\Delta_n}(S_t)$ satisfies (a) and (b) from \S\ref{Tabc}.
\end{proposition}

\begin{proof}  
Write $\mu_{\Delta_{n+1}}$ and $1_{\Delta_{n+1}}$ for the morphisms given by (\ref{multD}) and (\ref{unitD}) respectively.   First, it is easy to see that $m$ (resp. $\id_{\Delta_n}\otimes 1_1$) is a morphism of $\Delta_n$-modules.  Hence, $\mu_{\Delta_{n+1}}$ (resp. $1_{\Delta_{n+1}}$) is a morphism of $\Delta_n$-modules too.  Now, to show $\Delta_{n+1}$ satisfies (a) from \S\ref{Tabc} we must show the following equations hold in $\uRep^{\Delta_n}(S_t)$:
\begin{equation}\label{a}\begin{array}{c}
\mu_{\Delta_{n+1}}(\mu_{\Delta_{n+1}}\otimes_{\Delta_{n}}\id_{\Delta_{n+1}})=\mu_{\Delta_{n+1}}(\id_{\Delta_{n+1}}\otimes_{\Delta_{n}}\mu_{\Delta_{n+1}}),\\
\mu_{\Delta_{n+1}}(1_{\Delta_{n+1}}\otimes_{\Delta_{n}}\id_{\Delta_{n+1}})=\id_{\Delta_{n+1}}=\mu_{\Delta_{n+1}}(\id_{\Delta_{n+1}}\otimes_{\Delta_{n}}1_{\Delta_{n+1}}),\\
\mu_{\Delta_{n+1}}\beta_{\Delta_{n+1},\Delta_{n+1}}=\mu_{\Delta_{n+1}},
\end{array}
\end{equation} 
where $\beta_{\Delta_{n+1},\Delta_{n+1}}:\Delta_{n+1}\otimes_{\Delta_n}\Delta_{n+1}\to\Delta_{n+1}\otimes_{\Delta_n}\Delta_{n+1}$ denotes the braiding morphism.   To do so, first notice that by (\ref{ortho}) the morphisms $proj, inc$, and $\id_{\Delta_{n+1}}$ are all given by $x_{n+1}$.  
Let $\tau$ (resp. $\nu$) denote the identity morphism on $\Delta_{n+1}\otimes_{\Delta_n}\Delta_{n+1}$ (resp. $\Delta_{n+1}\otimes_{\Delta_n}\Delta_{n+1}\otimes_{\Delta_n}\Delta_{n+1}$).  Then, by the definition of $\otimes_{\Delta_n}$, we have the following realizations of $\tau$ and $\nu$ as morphisms in $\uRep_0(S_t)$:
 \begin{equation}\label{taunu}
\begin{array}{l}
\tau=(x_n\otimes\beta_{1,1})(x_{n+1}\otimes\id_{[pt]})(x_n\otimes\beta_{1,1})(x_{n+1}\otimes\id_{[pt]}),\\
\nu=(x_n\otimes\beta_{1,2})(x_{n+1}\otimes\id_{[pt]\otimes[pt]})(x_n\otimes\beta_{2,1})(\tau\otimes\id_{[pt]}),
\end{array}
\end{equation} where $\beta_{n,m}:A_n\otimes A_m\to A_m\otimes A_n$ denotes the braiding morphism in $\uRep_0(S_t)$ for each $n,m\geq 0$.
Moreover,  \begin{equation}\label{1mb}
1_{\Delta_{n+1}}=x_{n+1}(x_n\otimes 1_1),\quad\mu_{\Delta_{n+1}}=x_{n+1}(x_n\otimes\mu_1)\tau,\quad \beta_{\Delta_{n+1},\Delta_{n+1}}=\tau(x_n\otimes\beta_{1,1})\tau.
\end{equation}
%
%
Thus, showing the equations in (\ref{a}) hold in $\uRep^{\Delta_n}(S_t)$ amounts to showing the following equations hold in $\uRep_0(S_t)$:
\begin{equation*}\label{azero}
\begin{array}{l}
x_{n+1}(x_n\otimes\mu_1)\tau(x_n\otimes\beta_{1,1})(x_{n+1}\otimes\id_{[pt]})(x_n\otimes\beta_{1,1})((x_{n+1}(x_n\otimes\mu_1)\tau)\otimes\id_{[pt]})\nu=\\
x_{n+1}(x_n\otimes\mu_1)\tau(x_n\otimes\beta_{1,1})((x_{n+1}(x_n\otimes\mu_1)\tau)\otimes\id_{[pt]})(x_n\otimes\beta_{1,2})(x_{n+1}\otimes\id_{[pt]\otimes[pt]})\nu,\\
x_{n+1}(x_n\otimes\mu_1)\tau(x_n\otimes\beta_{1,1})(x_{n+1}\otimes\id_{[pt]})(x_n\otimes\beta_{1,1})(x_{n+1}(x_n\otimes 1_1)\otimes\id_{[pt]})\\
=x_{n+1}=
x_{n+1}(x_n\otimes\mu_1)\tau(x_{n+1}\otimes\beta_{1,1})((x_{n+1}(x_n\otimes 1_1))\otimes\id_{[pt]})x_{n+1},\\
x_{n+1}(x_n\otimes\mu_1)\tau(x_n\otimes\beta_{1,1})\tau=x_{n+1}(x_n\otimes\mu_1)\tau.
\end{array}
\end{equation*}
All equations above are straightforward to check using Proposition \ref{polyt} and \cite[Theorem 2.6 and Equation (2.6)]{CO}.  Thus $\Delta_{n+1}$ satisfies part (a) from \S\ref{Tabc}.

To show $\Delta_{n+1}$ satisfies part (b) from \S\ref{Tabc}, first notice that $\Delta_{n+1}\in\uRep^{\Delta_n}(S_t)$ is self dual (because the morphism $x_{n+1}$ is self dual).   Hence, we are required to show that the following morphism is invertible in $\uRep^{\Delta_n}(S_t)$:
\begin{equation}\label{b}
((\Tr~\mu_{\Delta_{n+1}})\otimes_{\Delta_n}\id_{\Delta_{n+1}})(\id_{\Delta_{n+1}}\otimes_{\Delta_n}coev_{\Delta_{n+1}}):\Delta_{n+1}\to\Delta_{n+1},
\end{equation}
where the morphism $\Tr:\Delta_{n+1}\to\Delta_n$ is defined in \S\ref{Tabc}(b).  In fact, we claim the morphism in (\ref{b}) is equal to the identity morphism $\id_{\Delta_{n+1}}$.  To prove this claim, first notice that 
\begin{equation}
\Tr=ev_{\Delta_{n+1}}\beta_{\Delta_{n+1}, \Delta_{n+1}}(\mu_{\Delta_{n+1}}\otimes_{\Delta_{n}}\id_{\Delta_{n+1}})(\id_{\Delta_{n+1}}\otimes_{\Delta_{n}}coev_{\Delta_{n+1}}).
\end{equation}
Also, $ev_{\Delta_{n+1}}=x_n(x_n\otimes ev_{[pt]})\tau$ and $coev_{\Delta_{n+1}}=\tau(x_n\otimes coev_{[pt]})x_n$.  Hence, using (\ref{taunu}), (\ref{1mb}), and the definition of $\otimes_{\Delta_n}$, we can realize the morphism in (\ref{b}) as a morphism in $\uRep_0(S_t)$.  Now use Proposition \ref{polyt} and \cite[Theorem 2.6 and Equation (2.6)]{CO} to show that this morphism is equal to $x_{n+1}$.
\end{proof}

\subsection{Deligne's lemma}\label{DLsec} Fix an integer $d\geq0$.   Set $\Delta=\Delta_{d+1}\in\uRep(S_d)$ and $\Delta^+=\Delta_{d+2}\in\uRep^\Delta(S_d)$.   By Proposition \ref{Deltadim}, $\dim_{\uRep^\Delta(S_d)}(\Delta^+)=-1$.  Hence, by Propositions \ref{uRepuni} and \ref{abDelta}, there exists a tensor functor $\cat{F}_\Delta:\uRep(S_{-1})\to\uRep^\Delta(S_d)$ with $\cat{F}_\Delta([pt])=\Delta^+$. Let $\uRes^{S_d}_{S_{-1}}$   denote the tensor functor $\uRep(S_d)\to\uRep(S_{-1})$ described in \S \ref{222}, i.e.~the functor prescribed by Proposition \ref{uRepuni} with $\uRes_{S_{-1}}^{S_d}([pt])=[pt]\oplus[\varnothing]^{\oplus d+1}$.  Then we have the following 

\begin{lemma}\label{Delemma} The functor $\Delta\otimes-:\uRep(S_d)\to\uRep^\Delta(S_d)$ is isomorphic to the composition $\cat{F}_\Delta\circ\uRes^{S_d}_{S_{-1}}$.
\end{lemma}   

\begin{proof} Both $\Delta\otimes-$ and $\cat{F}_\Delta\circ\uRes^{S_d}_{S_{-1}}$ are tensor functors which map $[pt]\in\uRep(S_d)$ to an object isomorphic to $\Delta^+\oplus\Delta^{\oplus d+1}\in\uRep^\Delta(S_d)$ (see Propositions \ref{Deltaj}(2) and \ref{Deltapt}).  Hence, by Proposition \ref{uRepuni}, they are isomorphic.
\end{proof}

The following corollary to Deligne's lemma will be used in the next section to classify tensor ideals in $\uRep(S_d)$.

\begin{corollary}\label{idid} Every nonzero tensor ideal in $\uRep(S_d)$ contains a nonzero identity morphism.
\end{corollary}

\begin{proof} Suppose $\cat{I}$ is a nonzero tensor ideal in $\uRep(S_d)$. 
Since tensor ideals are closed under composition, it suffices to show that $\cat{I}$ contains a 
morphism which has a nonzero isomorphism as a direct summand.   Let $f$ be a nonzero morphism in $\cat{I}$.  Then, by Remark \ref{tensornonzero}, $\id_\Delta\otimes f$ is also a nonzero morphism in $\cat{I}$.  By Lemma \ref{Delemma}  $\id_\Delta\otimes f=\cat{F}_\Delta(f')$ for some nonzero morphism $f'$ in $\uRep(S_{-1})$.  Since $\uRep(S_{-1})$ is semisimple  (see \cite[Th\'eor\`eme 2.18]{Del07} or \cite[Corollary 5.21]{CO}) it follows that $f'$ (and therefore $\cat{F}_\Delta(f')$) is the direct sum of isomorphisms and zero morphisms. 
\end{proof}

\subsection{Tensor ideals in $\uRep(S_d)$} In this section we use results from \cite{CO} along with Corollary \ref{idid} to classify tensor ideals in $\uRep(S_d)$ for arbitrary $d\in\Z_{\geq0}$.\footnote{If $t\not\in\Z_{\geq0}$ then $\uRep(S_t)$ is semisimple (see \cite[Th\'eor\`eme 2.18]{Del07} or \cite[Corollary 5.21]{CO}).  Hence there are no nonzero proper tensor ideals in $\uRep(S_t)$ when $t\not\in\Z_{\geq0}$.}    We begin by introducing an equivalence class on Young diagrams:

\begin{definition} Consider the weakest equivalence relation on the set of all Young diagrams such that $\lambda$ and $\mu$ are equivalent 
whenever the indecomposable object $L(\lambda)$ is a direct summand of $L(\mu)\otimes[pt]$ in $\uRep(S_d)$.  When $\lambda$ and $\mu$ are in the same equivalence class we write $\lambda\stackrel{d}{\approx}\mu$.
\end{definition}

The following proposition contains enough information on the equivalence relation $\stackrel{d}{\approx}$ for us to classify tensor ideals in $\uRep(S_d)$.

\begin{proposition}\label{green}  Assume $d$ is a nonnegative integer and $\lambda, \mu$ are Young diagrams.

(1) A nonzero morphism of the form $L(\lambda)\to L(\mu)$ is a negligible morphism in $\uRep(S_d)$ if and only if $L(\lambda)$ or $L(\mu)$ is not the minimal indecomposable object in an infinite block of $\uRep(S_d)$.

(2) $\lambda\stackrel{d}{\approx}\mu$ whenever $L(\lambda)$ and $L(\mu)$ are in trivial blocks of $\uRep(S_d)$.

(3) $\lambda\stackrel{d}{\approx}\mu$ whenever $L(\lambda)$ is a non-minimal indecomposable object in an infinite block and $L(\mu)$ is in a trivial block of $\uRep(S_d)$.

(4) $\lambda\stackrel{d}{\approx}\mu$ whenever neither $L(\lambda)$ nor $L(\mu)$ is a minimal indecomposable object in an infinite block of $\uRep(S_d)$.

(5)  Suppose $\lambda\stackrel{d}{\approx}\mu$ and $\cat{I}$ is a tensor ideal in  $\uRep(S_d)$ containing $\id_{L(\lambda)}$.  Then $\id_{L(\mu)}$ is also in $\cat{I}$.

\end{proposition}

\begin{proof}  Part (1) follows from \cite[Proposition 3.25, Corollary 5.9, and Theorem 6.10]{CO}.  Part (2) is easy to check using \cite[Propositions 3.12,  5.15 and Lemma 5.20(1)]{CO}.  Part (4) follows from parts (2) and (3).  Part (5) is easy to check.  Hence, it suffices to prove part (3).  To do so, let $\mathsf{b}$ denote the infinite block of $\uRep(S_d)$ containing $L(\lambda)$.  We will proceed by induction on $\mathsf{b}$ with respect to $\prec$ (see \cite[Definition 5.12]{CO}).

If $\mathsf{b}$ is the minimal with respect to $\prec$, then using \cite[Proposition 3.12 and Lemmas 5.18(1) and 5.20(1)]{CO} we can find a Young diagram $\rho$ with $L(\rho)$ in a trivial block of $\uRep(S_d)$ such that $\lambda\stackrel{d}{\approx}\rho$.  By part (2) $\rho\stackrel{d}{\approx}\mu$ and we are done.    Now suppose $\mathsf{b}$ is not minimal with respect to $\prec$.  Then, using \cite[Proposition 3.12 and Lemmas 5.18(2) and 5.20(2)]{CO}, we can find a Young diagram $\rho'$ with $\lambda\stackrel{d}{\approx}\rho'$ such that $L(\rho')$ is in an infinite block $\mathsf{b}'$ of $\uRep(S_d)$ with $\mathsf{b}'\precneqq \mathsf{b}$.  By induction $\rho'\stackrel{d}{\approx}\mu$ and we are done.
\end{proof}

We are now ready to classify tensor ideals in $\uRep(S_d)$.

\begin{theorem}\label{idealclassif}
If $d$ is a nonnegative integer, then the only nonzero proper tensor ideal in $\uRep(S_d)$ is the ideal of negligible morphisms. 
\end{theorem}

\begin{proof}  Assume $\cat{I}$ is a nonzero proper tensor ideal of $\uRep(S_d)$.  Then $\cat{I}$ is contained in the ideal of negligible morphisms (see \cite[Proposition 3.1]{GW}), hence we must show that $\cat{I}$ contains all negligible morphisms.  Suppose $\lambda$ is a Young diagram such that $L(\lambda)$ is not the minimal indecomposable object in an infinite block of $\uRep(S_d)$.  By  Proposition \ref{green}(1), it suffices to show $\id_{L(\lambda)}$ is contained in $\cat{I}$.
By Corollary \ref{idid}, there exists a nonzero identity morphism in $\cat{I}$.  It follows that $\cat{I}$ contains $\id_{L(\mu)}$ for some Young diagram $\mu$.  In particular, $\id_{L(\mu)}$ is a negligible morphism.  Hence, by Proposition \ref{green}(1), $L(\mu)$ is not the minimal indecomposable object in an infinite block of $\uRep(S_d)$.  Thus, by Proposition \ref{green}(4), $\lambda\stackrel{d}{\approx}\mu$.  Finally, by  Proposition \ref{green}(5), $\id_{L(\lambda)}$ is contained in $\cat{I}$.
\end{proof}

\begin{corollary}\label{DeltaGen} The tensor ideal in $\uRep(S_d)$ generated by $\id_\Delta$ is the ideal of all negligible morphisms.
\end{corollary}

\begin{proof} $\id_\Delta=x_{d+1}$ is a nonzero negligible morphism in $\uRep(S_d)$ (see \cite[Remark 3.22]{CO}).  Hence, the result follows from Theorem \ref{idealclassif}.
\end{proof}

\section{$t-$structure on $K^b(\uRep(S_d))$}
\subsection{Homotopy category} \label{homgen}
Let $\A$ be an additive category. Let $K^b(\A)$ be the bounded homotopy category of $\A$, see e.g.~\cite[\S 11]{KS}. Thus the objects of $K^b(\A)$ are finite complexes of objects in $\A$ and the morphisms are morphisms of  complexes up to homotopy. The category $K^b(\A)$ has a natural
structure of a triangulated category, see {\em loc.~cit.} In particular, for each integer $n$ 
we have a translation functor  $[n]: K^b(\A)\to K^b(\A)$.

Any object $A\in \A$ can be considered as a complex $A[0]$ concentrated in degree 0 or,
more generally, as a complex $A[n]$ concentrated in degree $-n$.
Thus we have a fully faithful functor $\A \to K^b(\A)$, $A\mapsto A[0]$. We will say that an
object $K\in K^b(\A)$ is {\em split} if it is isomorphic to an object of the form
$\oplus_iA_i[n_i]$ with $A_i\in \A$, $n_i\in \BZ$.

Now assume that $\A$ is an additive tensor category. 
The category $K^b(\A)$ has a natural structure of an additive tensor category. 
If the category $\A$ is braided or symmetric then so is the category $K^b(\A)$.
The functor $\A \to K^b(\A)$, $A\mapsto A[0]$ has an obvious structure of a (braided) tensor functor. 
If the category $\A$ is rigid so is the category $K^b(\A)$.

\subsection{Definition of $t-$structure} \label{deft}
We can apply the construction from \S \ref{homgen} to the case $\A =\uRep(S_d)$. 
We obtain a triangulated tensor category $\K_d:=K^b(\uRep(S_d))$.
We have

\begin{proposition} \label{split}
 For any $K\in \K_d$ the object $\Delta \ot K$ is split.
\end{proposition}

\begin{proof} By Lemma \ref{Delemma}, the functor $\Delta \ot -: \uRep(S_d)\to \uRep(S_d)$ 
is naturally isomorphic to a composition $\uRep(S_d)\to \uRep(S_{-1})\to \uRep(S_d)$. 
The category $\uRep(S_{-1})$ is semisimple (\cite[Th\'eor\`eme 2.18]{Del07} or \cite[Corollary 5.21]{CO}), so every object of $K^b(\uRep(S_{-1}))$ is
split. The result follows.
\end{proof}

We define $\K_d^{\le 0}$ as the full subcategory of $\K_d$ consisting of objects $K$ such that
$\Delta\ot K$ is concentrated in non-positive degrees (that is isomorphic to
$\oplus_iA_i[n_i]$ with $A_i\in \A$ and $n_i\in \BZ_{\ge 0}$). Similarly, we define
$\K_d^{\ge 0}$ as the full subcategory of $\K_d$ consisting of objects $K$ such that
$\Delta\ot K$ is concentrated in non-negative degrees. The following result
will be proved in \S \ref{axiom}.

\begin{theorem} \label{tstrth}
The pair $(\K_d^{\le 0}, \K_d^{\ge 0})$ is a $t-$structure {\em (see \cite[D\'efinition 1.3.1]{BBD})} on the category $\K_d$.
\end{theorem} 

Recall that the core of this $t-$structure is the subcategory $\K_d^0=\K_d^{\le 0}\cap \K_d^{\ge 0}$. By
definition this means that $K\in \K_d^0$ if and only if $\Delta\ot K$ is concentrated in degree zero.
In particular, for any $A\in \uRep(S_d)$ the object $A[0]\in \K_d^0$. We have

\begin{corollary}\label{Kdabtens} (a) The category $\K_d^0$ is abelian.

(b) The category $\K_d^0$ is a tensor subcategory of $\K_d$.
\end{corollary}

\begin{proof} (a) follows from Theorem \ref{tstrth} and \cite[Th\'eor\`eme 1.3.6]{BBD}. 
For (b) we need to check that for $K,K' \in \K_d^0$
we have $K\ot K'\in \K_d^0$. Assume this is not the case. This means that the split complex
$\Delta \ot K\ot K'$ is not concentrated in degree zero. Since $\Delta \ot X\ne 0$ for any
$0\ne X\in \uRep(S_d)$ (see Remark \ref{tensornonzero}) we get that $\Delta \ot \Delta \ot K\ot K'$ is split and not
concentrated in degree zero. But this is not the case since $\Delta \ot \Delta \ot K\ot K'\simeq 
(\Delta \ot K)\ot (\Delta \ot K')$ and both $\Delta \ot K$ and $\Delta \ot K'$ are split and concentrated
in degree zero. 
\end{proof}

We will show in \S \ref{axiom} that the category $\K_d^0$ is actually pre-Tannakian. Thus we
constructed a fully faithful tensor functor $\uRep(S_d)\to \K_d^0$ where $\K_d^0$ is a pre-Tannakian
category. Of course a priori this might be quite different from Deligne's functor
$\uRep(S_d)\to \uRep^{ab}(S_d)$. 

\subsection{Verification of $t-$structure axioms}\label{axiom}
The main goal of this Section is to prove Theorem \ref{tstrth}.
\subsubsection{}\label{tneg}
 We start by reformulating  the definition of
$\K_d^{\le 0}$ and $\K_d^{\ge 0}$ in terms of negligible objects, i.e. objects whose identity morphisms are negligible.  

\begin{proposition} \label{vanish}
 Let $K\in \K_d$. Then $K\in \K_d^{\le 0}$ if and only if $\Hom(K,A[n])=0$
for any negligible $A\in\uRep(S_d)$ and $n\in \BZ_{<0}$. Similarly, $K\in \K_d^{\ge 0}$ if and only if $\Hom(K,A[n])=0$
for any negligible $A\in\uRep(S_d)$ and $n\in \BZ_{>0}$.
\end{proposition}

\begin{proof} We prove only the characterization of $\K_d^{\le 0}$ (the case of $\K_d^{\ge 0}$ is similar).
Assume first that $\Hom(K,A[n])=0$ for any negligible $A$ and $n\in \BZ_{<0}$. By Proposition \ref{xidemp} $\Delta^\ast=\Delta$, thus by Corollary \ref{DeltaGen} $\Delta\otimes B=\Delta^\ast\otimes B$ is negligible for all $B\in\uRep(S_d)$.  Hence,  $\Hom(\Delta \ot K, B[n])=\Hom(K, \Delta^*\ot B[n])=0$ for any $B\in\uRep(S_d)$ and $n\in \BZ_{<0}$.
Since by Proposition \ref{split} the object $\Delta \ot K\in \K_d$ is split we get immediately that
$K\in \K_d^{\le 0}$.

Conversely, assume that $K\in \K_d^{\le 0}$. Then by definition $\Hom(\Delta \ot K, B[n])=0$ for any
$B\in \uRep(S_d)$ and $n\in \BZ_{<0}$. Hence $\Hom(K, \Delta^* \ot B[n])=0$. Since, by Corollary \ref{DeltaGen}, any
negligible object is a direct summand of an object of the form $\Delta\ot B=\Delta^*\ot B$ we are done.
\end{proof}

\subsubsection{Blockwise description of $(\K_d^{\le 0}, \K_d^{\ge 0})$}\label{tblock}
Recall that the category $\uRep(S_d)$ decomposes into a direct sum of blocks
$\uRep(S_d)=\oplus_\mathsf{b} \uRep_{\mathsf{b}}(S_d)$, see \S \ref{StBlocks}. Similarly, we have a
decomposition $\K_d=\oplus_\mathsf{b}(\K_d)_\mathsf{b}$ (in other words, for any object $K\in \K_d$ 
we have a canonical decomposition $K=\oplus_\mathsf{b}K_\mathsf{b}$ where
all the terms of the complex $K_\mathsf{b}\in (\K_d)_\mathsf{b}$ are in the block $\uRep_\mathsf{b}(S_d)$). 
Since $\Delta \ot (\oplus_\mathsf{b}K_\mathsf{b})=\oplus_\mathsf{b}\Delta \ot K_\mathsf{b}$ we see that $K=\oplus_\mathsf{b}K_\mathsf{b}\in \K_d^{\le 0}$
if and only if $K_\mathsf{b}\in \K_d^{\le 0}$ for any $\mathsf{b}$ (and similarly for $\K_d^{\ge 0}$). 
In other words $\K_d^{\le 0}=\oplus_\mathsf{b}(\K_d^{\le 0})_\mathsf{b}$ where $(\K_d^{\le 0})_\mathsf{b}=\K_d^{\le 0}\cap (\K_d)_\mathsf{b}$,
that is the subcategory $\K_d^{\le 0}$ is compatible with the block decomposition (and similarly
for $\K_d^{\ge 0}=\oplus_\mathsf{b}(\K_d^{\ge 0})_\mathsf{b}$). Thus in order to
verify that $(\K_d^{\le 0}, \K_d^{\ge 0})$ is a $t-$structure on $\K_d$ it is sufficient to verify
that $((\K_d^{\le 0})_\mathsf{b}, (\K_d^{\ge 0})_\mathsf{b})$ is a $t-$structure on $(\K_d)_\mathsf{b}$ for every block $\mathsf{b}$. 
Fortunately, Proposition \ref{vanish} gives rise to an easy description of $(\K_d^{\le 0})_\mathsf{b}$
and $(\K_d^{\ge 0})_\mathsf{b}$.

\begin{proposition} \label{blockbyblock}
Let $K\in (\K_d)_\mathsf{b}$.

 (a) Assume that $\mathsf{b}$ is a semisimple block and let $L$ be a unique indecomposable object in $\mathsf{b}$.
  Then $K\in (\K_d^{\le 0})_\mathsf{b}$ (resp. $K\in (\K_d^{\ge 0})_\mathsf{b}$) if and only if 
$K\in (\K_d)_\mathsf{b}$ and $\Hom(K,L[n])=0$
   for any $n\in \BZ_{<0}$ (resp. for $n\in \BZ_{>0}$).

(b) Assume that $\mathsf{b}$ is a non-semisimple block with indecomposable objects $L_i$ for $i\in \BZ_{\ge 0}$
labeled as in Proposition \ref{StBlockProp}(ii). 
Then $K\in (\K_d^{\le 0})_\mathsf{b}$ (resp. $K\in (\K_d^{\ge 0})_\mathsf{b}$) if and only if 
$K\in (\K_d)_\mathsf{b}$ and $\Hom(K,L_i[n])=0$
 for all $i>0$ and any $n\in \BZ_{<0}$ (resp. for $n\in \BZ_{>0}$).
\end{proposition}

\begin{proof} Combine Proposition \ref{vanish} and Proposition \ref{StBlockProp}.
\end{proof}

\subsubsection{Analogy with Temperley-Lieb category} 
The definition of the $t-$structure in \S \ref{deft} was motivated by the following analogy. Pick a nontrivial root of unity $q$ such that
$q+q^{-1}\in F$ and recall the Temperley-Lieb category $TL(q)$ from \S \ref{TLcat}. 
Consider the category $K^b(TL(q))$. It is well known (see e.g.~\cite[Proposition 2.7]{O} that the embedding $TL(q)\subset \C_q$
induces an equivalence of triangulated categories $K^b(TL(q))\simeq D^b(\C_q)$ where $D^b(\C_q)$ is the derived category
of the abelian category $\C_q$. In particular the category $\D_q:=K^b(TL(q))$ inherits a natural $t-$structure 
$(\D_q^{\le 0},\D_q^{\ge 0})$ from the category $D^b(\C_q)$, see e.g.~\cite[Exemple 1.3.2(i)]{BBD}\footnote{Thus the category 
$\D_q^{\le 0}$ consists of objects of $D^b(\C_q)$ with nontrivial cohomology only in non-positive degrees and similarly for
$\D_q^{\ge 0}$.}. 
This $t-$structure can be characterized as follows. 

Let $St:=V_{l-1}\in TL(q)$ be the so called {\em Steinberg module}. It is known 
(see \cite[Theorem 9.8]{APW}) that $St$ is a projective
object of the category $\C_q$. Thus $St\ot Y$ is a projective object of $\C_q$ for any $Y\in \C_q$, see 
\cite[Lemma 9.10]{APW}.
In particular, for any $K\in \D_q$ the object $St\ot K\in \D_q$ is isomorphic to its cohomology
(as a finite complex consisting of projective modules and with projective cohomology).   It is well known that each projective
object of $\C_q$ is contained in $TL(q)\subset \C_q$,
see \cite[(5.7)]{A}. Thus in the language of \S \ref{homgen} for any $K\in K^b(TL(q))$ 
the complex $St\ot K$ is split (analogous to Proposition \ref{split}). It is clear that $K\in \D_q^{\le 0}$ if and only if
$St\ot K$ is concentrated in non-positive degrees and similarly for $\D_q^{\ge 0}$. This is a counterpart of the definition
of the $t-$structure $(\K_d^{\le 0},\K_d^{\ge 0})$.

Furthermore, it is known that each direct summand of $St\ot Y$ for $Y\in TL(q)$ is negligible
(see \cite[Proposition 3.5 and Lemma 3.6]{A}) and that each negligible
object of $TL(q)$ is a direct summand of $St\ot Y$ with $Y\in TL(q)$, see \cite[p. 158]{A}.
 Thus we have the following counterpart
of Proposition \ref{vanish} (with a similar proof):

(a) {\em  Let $K\in \D_q$. Then $K\in \D_q^{\le 0}$ (resp. $K\in \D_q^{\ge 0}$) if and only if $\Hom(K,A[n])=0$
for any negligible $A\in TL(q)$ and $n\in \BZ_{<0}$ (resp. $n\in \BZ_{>0}$).} 

Hence, following \S \ref{tblock}, we can give a blockwise description of the $t-$structure $(\D_q^{\le 0},\D_q^{\ge 0})$.
For a block $\mathsf{b}$ let $(\D_q)_{\mathsf{b}}$ denote the full subcategory of $\D_q=K^b(TL(q))$ consisting of
complexes with all terms from the block $\mathsf{b}$. 
Using Lemma \ref{TLneglig} we obtain the following counterpart of Proposition \ref{blockbyblock}:

(b) {\em  Let $\mathsf{b}$ be a non-semisimple block of $TL(q)$ with indecomposable objects $L_i$ for $i\in \BZ_{\ge 0}$
labeled as in Remark \ref{TLlabel}.  Let $K\in (\D_q)_{\mathsf{b}}$.
Then $K\in \D_q^{\le 0}$ (resp. $K\in \D_q^{\ge 0}$) if and only if $\Hom(K,L_i[n])=0$
 for all $i>0$ and any $n\in \BZ_{<0}$ (resp. for $n\in \BZ_{>0}$).}

From this description it is clear that the pair $(\D_q^{\le 0}\cap (\D_q)_{\mathsf{b}}, \D_q^{\ge 0}\cap (\D_q)_{\mathsf{b}})$
of subcategories of $(\D_q)_{\mathsf{b}}$ corresponds to the pair 
$((\K_d^{\le 0})_{\mathsf{b}'}, (\K_d^{\ge 0})_{\mathsf{b}'})$ under the equivalence
$(\D_q)_{\mathsf{b}}\simeq (\K_d)_{\mathsf{b}'}$ induced by the equivalence of blocks from Proposition \ref{TLblocks}.
Since  $(\D_q^{\le 0}\cap (\D_q)_{\mathsf{b}}, \D_q^{\ge 0}\cap (\D_q)_{\mathsf{b}})$ is a $t-$structure on the category
$(\D_q)_{\mathsf{b}}$ we have the following 

\begin{corollary} \label{pokus}
Let $\mathsf{b}$ be a non-semisimple block of the category $TL(q)$ and let $\mathsf{b}'$ be an equivalent
block in the category $\uRep(S_d)$ as in Proposition \ref{TLblocks}. Then 
$((\K_d^{\le 0})_{\mathsf{b}'}, (\K_d^{\ge 0})_{\mathsf{b}'})$ is a $t-$structure on the category $(\K_d)_{\mathsf{b}'}$.
$\square$
\end{corollary}

\subsubsection{Proof of Theorem \ref{tstrth}}\label{prtstr}
 It suffices to show 
$((\K_d^{\le 0})_\mathsf{b}, (\K_d^{\ge 0})_\mathsf{b})$ is a $t-$structure on $(\K_d)_\mathsf{b}$ for every block $\mathsf{b}$. 
If the block $\mathsf{b}$ is semisimple then the category $(\K_d)_\mathsf{b}$ can be identified with $K^b(\Vec_F)$ and Proposition \ref{blockbyblock} (a) shows that
$((\K_d^{\le 0})_\mathsf{b}, (\K_d^{\ge 0})_\mathsf{b})$ is the standard $t-$structure on $K^b(\Vec_F)$.

It remains to consider the case when $\mathsf{b}$ is a non-semisimple block. 
Choose a nontrivial root of unity $q$ such that $q+q^{-1}\in F$ (for example a primitive cubic root of unity $\zeta$
will work for any $F$ since $\zeta+\zeta^{-1}=-1\in F$). Then there is a non-semisimple block in $TL(q)$, which is equivalent to $\mathsf{b}$ (Proposition \ref{TLblocks}).  Hence, by Corollary \ref{pokus}, $((\K_d^{\le 0})_\mathsf{b}, (\K_d^{\ge 0})_\mathsf{b})$ is a $t-$structure on $(\K_d)_\mathsf{b}$.
\hfill $\square$

\subsubsection{Complements} The proof in \S \ref{prtstr} implies the following

\begin{corollary}\label{KpreT}
 (a) The category $\K_d^0$ is pre-Tannakian.

(b) Any object of the category $\K_d^0$ is isomorphic to a subquotient of a direct sum of tensor powers 
of $[pt]$.
\end{corollary}

\begin{proof} We already know that the category $\K_d^0$ is an abelian tensor category (see Corollary \ref{Kdabtens}).
It is obvious that $\Hom$'s are finite dimensional and $\End(\be)=F$ since this is true in
the category $\K_d$. The category $\K_d^0$ is rigid: if $\Delta \ot K$ is concentrated in degree zero
then the same is true for $\Delta \ot K^*\simeq (\Delta \ot K)^*$. It remains to check that any object
of $\K_d^0$ has finite length. It is clear that we can verify this block by block. The result is
clear for semisimple blocks since by Proposition \ref{blockbyblock}(a) the core of the corresponding
$t-$structure identifies with $\Vec_F$. This is also clear for non-semisimple blocks since 
the corresponding $t-$structure (described in Proposition \ref{blockbyblock}) identifies with 
the $t-$structure on a block of the Temperley-Lieb category and the corresponding core has all objects
of finite length since this is true for the category $\C_q$. This proves (a).

For (b) we use the same argument as above: it is sufficient to verify the statement block by block.
Here the result is trivial for semisimple blocks and is known for non-semisimple ones since it is
known to hold for the category $\C_q$. 
\end{proof}

\begin{remark} Using similar techniques of importing known results about the category $\C_q$
to the category $\K_d^0$ we can obtain  detailed information about this category. In particular,
we see that the category $\K_d^0$ has enough projective objects; all indecomposable 
projective objects are direct summands of tensor powers of $[pt]$ (but powers of $[pt]$ are not
projective in general; for example $[pt]^{\otimes 0}=\be$ is not projective). Thus Corollary \ref{KpreT}(b)
can be improved: any object of the category $\K_d^0$ is isomorphic to a quotient of a direct sum 
of tensor powers of $[pt]$.
\end{remark}

\section{Universal property} 
\subsection{Extension property of the category $\K_d^0$} \label{extension}
We start with the following

\begin{proposition} \label{ext}
 Let $\T$ be a pre-Tannakian category and let $\cat{F}: \uRep(S_d)\to \T$
be a tensor functor. Assume that $\cat{F}(\Delta)\ne 0$. 
Then the functor $\cat{F}$ (uniquely) factorizes as
$\uRep(S_d)\to \K^0_d\to \T$ where $\K^0_d\to \T$ is an exact tensor functor.
\end{proposition}

\begin{proof} Let $K\in \K_d^0$. We can consider $\cat{F}(K)\in K^b(\T)$. Since the category $\T$
is abelian we can talk about cohomology of $\cat{F}(K)$. 

\begin{lemma} $H^i(\cat{F}(K))=0$ for $i\ne 0$.
\end{lemma}

\begin{proof} Notice that for any $0\ne X\in \T$ we have $X\ot \cat{F}(\Delta)\ne 0$.
Since the  endofunctor $-\ot \cat{F}(\Delta)$ of the category $\T$ is exact (see e.g.~\cite[Proposition 2.1.8]{BK}) we see
that $H^i(\cat{F}(K \ot \Delta))=H^i(\cat{F}(K)\ot \cat{F}(\Delta))=H^i(\cat{F}(K))\ot \cat{F}(\Delta)$. 
By definition of $\K_d^0$ the cohomology of $\cat{F}(K\ot \Delta)$ is concentrated in degree zero
and we are done.
\end{proof}

We now define the functor $\K^0_d\to \T$ as $K\mapsto H^0(\cat{F}(K))$ with the tensor structure induced by the 
one on $\cat{F}$ (or rather its extension to $K^b(\uRep(S_d))\to \K^b(\T)$). 
\end{proof}

\begin{remark} Here is an example of tensor functor between abelian rigid tensor categories which is not exact. Let $k$ be
a field of characteristic 2 and consider the category $\Rep_k(\BZ/2\BZ)$ of finite dimensional $k-$representations of $\BZ/2\BZ$.
This category has precisely 2 indecomposable objects: one is simple and 1-dimensional; the other is projective and has categorical dimension 0. Thus the 
quotient of $\Rep_k(\BZ/2\BZ)$ by the negligible morphisms is equivalent to the category $\Vec_k$ of finite dimensional vector
spaces over $k$. Clearly the quotient functor $\Rep_k(\BZ/2\BZ)\to \Vec_k$ is not exact since it sends the projective object
to zero. One can also construct a similar example over a field of characteristic zero using the representation category of the additive
supergroup of a 1-dimensional odd space.
\end{remark}

\subsection{Fundamental groups of $\K_d^0$ and $\Rep(S_d)$}
Let $\pi$ be the fundamental group of the pre-Tannakian category $\K_d^0$. The action of
$\pi$ on $[pt]\in \uRep(S_d)\subset \K_d^0$ defines a homomorphism $\pi \to S_\bI$ where
$\bI =\Spec([pt])$. 

\begin{proposition}\label{fgusd}
 The homomorphism $\eps: \pi \to S_\bI$ is in fact an isomorphism.
\end{proposition}

\begin{proof} Since the object $[pt]$ generates $\K_d^0$ (see Corollary \ref{KpreT}(b)) the homomorphism 
$\eps : \pi \to S_\bI$ is an embedding. 

Consider the category $\Rep(S_\bI,\eps)$. It is shown in (the proof of) \cite[Proposition B1]{Del07}
that its fundamental group is precisely the group $S_\bI^\eps=\Aut_{\Rep(S_\bI,\eps)}(\bI)$. We have an obvious
tensor functor $\uRep(S_d)\to \Rep(S_\bI,\eps)$; by Proposition \ref{ext} it extends to a tensor functor
$\cat{F}: \K_d^0\to \Rep(S_\bI,\eps)$. Thus we have a homomorphism $S_\bI^\eps \to \cat{F}(\pi)$. It is clear
that the composition $S_\bI^\eps \to \cat{F}(\pi)\subset \cat{F}(S_\bI)=S_\bI^\eps$ is the identity map. 
The result follows.
%
\end{proof}

We also recall here the following result from \cite[8.14(ii)]{Del90}:

\begin{proposition}\label{fgsd}
 The fundamental group of the category $\Rep(S_d)$ is the group $S_d$ acting
on itself by conjugation. $\square$
\end{proposition} 

\begin{remark} It is explained in \cite[8.14(ii)]{Del90} that we can replace $S_d$ with any affine algebraic group $G$ in the statement of the previous proposition.  
\end{remark} 

\subsection{Proof of Theorem \ref{main}} We start with the following result:

\begin{theorem} \label{mainly}
Let $\T$ be a pre-Tannakian category and let $\cat{F}: \uRep(S_d)\to \T$
be a tensor functor with $T=\cat{F}([pt])$.

(a) If $\cat{F}(\Delta)=0$ then the category 
$\Rep(S_\bI,\eps)$ endowed with the functor $\cat{F}_T: \uRep(S_d)\to \Rep(S_\bI,\eps)$ is equivalent
to $\Rep(S_d)$ equipped with the functor $\uRep(S_d)\to \Rep(S_d)$.

(b) If $\cat{F}(\Delta)\ne 0$ then the category 
$\Rep(S_\bI,\eps)$ endowed with the functor $\cat{F}_T: \uRep(S_d)\to \Rep(S_\bI,\eps)$ is equivalent
to $\K_d^0$ equipped with the functor $\uRep(S_d)\to \K_d^0$.
\end{theorem}

\begin{proof} (a) In this case  $\cat{F}$ factorizes as $\uRep(S_d)\to \Rep(S_d)\to \T$
(see Corollary \ref{DeltaGen}). The result follows from \cite[Theoreme 8.17]{Del90} and Proposition \ref{fgsd}.

(b) In this case $\cat{F}$ extends to a functor $\uRep(S_d)\to \K_d^0\to \T$ by Proposition
\ref{ext}. The result follows from \cite[Theoreme 8.17]{Del90} and Proposition \ref{fgusd}.
\end{proof}

If we apply Theorem \ref{mainly}(b) to the category $\T =\uRep(S_{-1})$ and the functor 
$\uRes^{S_d}_{S_{-1}}:\uRep(S_d)\to \T$ described in \S \ref{222} and  \S \ref{DLsec} we obtain the following

\begin{corollary} \label{abD}
The category $\uRep^{ab}(S_d)$ endowed with the functor $\uRep(S_d)\to
\uRep^{ab}(S_d)$ is equivalent to the category $\K_d^0$ with the functor $\uRep(S_d)\to \K_d^0$.
$\square$
\end{corollary}

Clearly Theorem \ref{mainly} and Corollary \ref{abD} together imply Theorem \ref{main}.

\bibliographystyle{ams-alpha}



\end{document}